\documentclass[12pt]{amsart}
\usepackage{times}
\usepackage[T1]{fontenc}
\usepackage{dsfont}
\usepackage{mathrsfs}
\usepackage{color}
\usepackage[a4paper,asymmetric]{geometry}
\usepackage{mathscinet}
\usepackage{latexsym}
\usepackage{amsthm}
\usepackage{amssymb}
\usepackage{amsfonts}
\usepackage{amsmath}
\newtheorem{theorem}{Theorem}[section]
\newtheorem{thm}[theorem]{Theorem}

\newtheorem{lem}[theorem]{Lemma}
\newtheorem{proposition}[theorem]{Proposition}

\newtheorem{corollary}[theorem]{Corollary}

\theoremstyle{definition}

\newtheorem{defn}[theorem]{Definition}

\theoremstyle{remark}

\newtheorem{rem}[theorem]{Remark}
\numberwithin{equation}{section}

 \DeclareMathAlphabet{\mathpzc}{OT1}{pzc}{m}{it}

  \newcommand{\dif}{\mathrm{d}}

 \newcommand{\e}{\varepsilon}
 \newcommand{\p}{\partial}

 \newcommand{\R}{\mathbb{R}}
 
 \newcommand{\PP}{\mathbb{P}}
 \newcommand{\mcl}{\mathcal}
 
 \newcommand{\Be}{\begin{equation}}
 \newcommand{\Ee}{\end{equation}}
 \newcommand{\Bs}{\begin{split}}
 \newcommand{\Es}{\end{split}}
  \newcommand{\Bes}{\begin{equation*}}
 \newcommand{\Ees}{\end{equation*}}
 \newcommand{\BT}{\begin{thm}}
 \newcommand{\ET}{\end{thm}}
 \newcommand{\Bp}{\begin{proof}}
 \newcommand{\Ep}{\end{proof}}
 \newcommand{\BL}{\begin{lem}}
 \newcommand{\EL}{\end{lem}}
 \newcommand{\BP}{\begin{proposition}}
 \newcommand{\EP}{\end{proposition}}
 \newcommand{\BC}{\begin{corollary}}
 \newcommand{\EC}{\end{corollary}}
 \newcommand{\BR}{\begin{rem}}
 \newcommand{\ER}{\end{rem}}
 \newcommand{\BD}{\begin{defn}}
 \newcommand{\ED}{\end{defn}}
 \newcommand{\BI}{\begin{itemize}}
 \newcommand{\EI}{\end{itemize}}

 \newcommand{\re}{{\rm e}}
 \newcommand{\Om}{\mcl O}

\begin{document}
\title[Blowup of stochastic Gierer-Meinhardt system]
{Finite time blowup of the stochastic shadow Gierer-Meinhardt System}

\author[F. Li]{Fang Li}
\address{Center for Partial Differential Equations, East China Normal University,
500 Dongchuan Road, Shanghai, 200241, China}
\email{fangli0214@gmail.com}

\author[L. Xu]{Lihu Xu}
\address{Faculty of Science and Technology,
University of Macau, E11
Avenida da Universidade, Taipa,
Macau, China}
\email{xulihu2007@gmail.com}

\thanks{FL is   supported by Chinese NSF (No. 11201148), Shanghai Pujiang Program (No. 13PJ1402400).}

\thanks{LX is supported by the grant SRG2013-00064-FST. He would like to thank the hospitality of
Center for Partial Differential Equations at ECNU and part of his work was done during visiting CPDE.}

\begin{abstract} \label{abstract}

By choosing some special (random) initial data, we prove that
with probability $1$,  the stochastic shadow Gierer-Meinhardt system blows up pointwisely in finite time.
We also give a (random)
upper bound for the blowup time and some estimates about this bound. By increasing the amplitude of the initial data, we can get
the blowup in any short time with positive probability.
\end{abstract}

\maketitle

\section{Introduction}
Many of the mathematical models that have been proposed for
the study of population dynamics, biochemistry, morphogenesis and
other fields, take the following form:
\begin{equation} \label{e:GenRD}
\left\{
\begin{array}{ll}
\p_t u=d_1\Delta u+f(u,v)  &\mbox{ in } \mcl O \times (0,T),\\[3mm]
\tau \p_t v =d_2\Delta v+g(u,v)  &\mbox{ in } \mcl O \times (0,T),\\[3mm]
\partial_\nu u=\p_\nu v=0 & \mbox{ on } \p \mcl O \times (0,T),
\end{array} \right. \end{equation}
where $\Delta=\sum_{i=1}^n \p^2_{x_i}$ is Laplace operator, $\mcl O$ is a bounded smooth
domain in $\R^n$ with unit outward normal vector $\nu$ on its boundary $\p \mcl O$; the two
positive constants $d_1,d_2$ are the diffusion rates of two substances $u$ and $v$ respectively;
$\tau>0$ is the number tuning response rate of $v$ related to the change of $u$; $f,g$ are both smooth
functions referred to as the reaction terms.

As we choose
\
\Be \label{e:FGCho}
f(u,v)=\frac{u^p}{v^q}, \ \ \ g(u,v)=\frac {u^r}{v^s},
\Ee
with $p>0, q>0, r>0, s \ge 0$ satisfying the condition:
\
\Be \label{e:PQRel}
0<\frac{p-1}{r}<\frac{q}{s+1},
\Ee
Eq. \eqref{e:GenRD} is the well known Gierer-Meinhardt system.
When '$\Delta$'s are removed, the corresponding ODEs have a stable
equilibrium solution $(1,1)$. The condition \eqref{e:PQRel} is imposed
so that $(1,1)$ becomes \emph{unstable} due to the two diffusion terms with
$d_1$ small and $d_2$ large. This idea was proposed by Turing in 1952
and used to explain the onset of pattern formation by an instability of an unpatterned state leading to a pattern.
It is now commonly called {\it Turing diffusion-driven
instability} (\cite{t1}). Based on this idea, the Gierer-Meinhardt system \eqref{e:GenRD}-\eqref{e:PQRel} was formulated in 1972 \cite{gm1} to model the regeneration phenomena of {\it hydra}.

The  Gierer-Meinhardt  system \eqref{e:GenRD}-\eqref{e:PQRel} is usually called \emph{full} system, its dynamics remains far from being   understood at this time.
First result in this direction was due to Rothe in 1984 \cite{r1}, but only for a very special case $n=3$, $p=2$, $q=1$, $r=2$ and $s=0$.  In 1987, a result for a related system was obtained in \cite{MT}.  The nearly optimal resolution for the
global existence issue came in 2006 with an elementary and elegant proof by Jiang \cite{j1}. In \cite{j1}, the global existence   was established for the range ${p-1\over r} < 1$. This only leaves the critical case ${p-1\over r} = 1$ still open, since it has been known already that in case ${p-1\over r} >1$
blow-up can occur   even   for the corresponding kinetic system (\cite{nst}).
\vskip 3mm

 When $d_2 \rightarrow \infty$, we expect that $v$ tends to be space-homogeneous,
 i.e., $v(x,t)$ will be a spatially constant but time dependent function $\xi(t)$. Now the above Gierer-Meinhardt system
 is replaced by the following \emph{shadow system}:
 \
 \begin{equation} \label{gmshadow0} \left\{
\begin{array}{ll}
\p_t u=d_1\Delta u-u+\frac{u^p}{\xi^q} & \mbox{ in } \mcl O \times (0,T),\\[2mm]
\tau \frac{\dif \xi}{\dif t}=\left(-\xi+\frac{\overline{u^r}}{\xi^s}\right) & \mbox{ in }  (0,T), \\[2mm]
\frac{\partial u}{\partial\nu}=0 & \mbox{ on }\partial\Om \times (0,T),
\end{array} \right. \end{equation}
where $\overline{u^r}=\frac{1}{|\Om|}\int_{\Om}u^r \, \dif x$ with $|\Om|$ being the volume of $\Om$.   This idea was suggested by Keener (\cite{k1})  and the name "shadow system'' was proposed by Nishiura (\cite{nish2}).

The dynamics  \eqref{gmshadow0} has been less well studied than the full Gierer-Meinhardt system.
Global existence and finite-time blow-up have firstly been explored by the first author and Ni (\cite{ln1}) in 2009.
In particular, they show that for $\frac{p-1}{r}<\frac{2}{n+2}$ there is a unique global solution, whereas for $\frac{p-1}{r}>\frac{2}{n}$ blow-up can occur provided that $p=r$, $\tau=s+1-q$. Later, Phan showed that Eq. \eqref{gmshadow0} also admits a global solution in the case $\frac{p-1}{r}=\frac{2}{n+2}$ (\cite{Pha12}). The first author and Yip continue the work in \cite{ln1} and improve the earlier results concerning blowup solutions to the optimal case  $\frac{p-1}{r}>\frac{2}{n+2}$.

Since the existence and blowup of the
solutions do not depend on the numbers $d_1$ and $\tau$. Without loss of generality, we shall assume $d_1=1$ and $\tau=1$ throughout the rest of  this paper.

The purpose of this paper is to  study the shadow Gierer-Meinhardt system with random migrations
with the following form:
\begin{equation} \label{gmshadow} \left\{
\begin{array}{ll}
\p_t u=\Delta u-u+\frac{u^p}{\xi^q} & \mbox{ in } \mcl O \times (0,T),\\[2mm]
{\dif \xi}=\left(-\xi+\frac{\overline{u^r}}{\xi^s}\right) \dif t+\xi \dif B_t & \mbox{ in }  (0,T), \\[2mm]
\frac{\partial u}{\partial\nu}=0 & \mbox{ on }\partial\Om \times (0,T), \\[2mm]
u(0)=u_0 & \mbox{ in }\Om, \\ [2mm]
  \xi(0)=\xi_0,
\end{array} \right. \end{equation}
where $\xi \dif B_t$ can be explained as random migrations and $B_t$ is a one-dimensional standard Brownian motion.
Due to the random
effects, we need to introduce the sample space $\Omega$ and re-define
$$u(t,x,\omega): \R^+ \times \mcl O \times \Omega \rightarrow \R^+, \ \ \ \xi(t,\omega): \R^+ \times \Omega \rightarrow \R^+\setminus \{0\}. $$

To our knowledge, there seem only two papers in the research of stochastic Gierer-Meinhardt type systems.
One is \cite{ks10}, which studied a system including two coupled stochastic PDEs with bounded and Lipschitz nonlinearity.
\cite{ks10} only proved the \emph{local} existence of the \emph{positive} stochastic solution by Da Prato-Zabczyk's approach (\cite{dpz92}).
The other is \cite{WXZZ14} established the \emph{global} existence of the strong positive solution and the large deviation principle
for Eq. \eqref{gmshadow}.
\vskip 3mm

 We shall study in this paper the blowup problem of Eq. \eqref{gmshadow} under quite general assumptions.
When $p \ge r$ and $\frac{p-1}{r}>\frac{2}{n+2}$, we show that with probability $1$, Eq. \eqref{gmshadow}
blows up pointwisely if we choose some suitable (random) initial data. We also give a (random)
upper bound for the blow up time and consequently obtain a probabilistic estimate of this blow up.

To our knowledge, there are not many results for the blow up of stochastic systems.
The work \cite{Cho09} proved that the 2nd moment of the solution of some nonlinear
wave equations blow up, while \cite{Cho11} gave a nice criterion for the blow up of some
stochastic reaction-diffusion equations under $p$th moments. As pointed out in \cite{Cho11},
the blowup under $p$th moments even does not imply the pathwise blowup with a
positive probability. \cite{ChLi12} extended the result in \cite{Cho11} to the case of stochastic parabolic
equations with delay. Most recently, Chow and Khasminski established an almost sure
blowup result for a family of SDEs (\cite{ChKh14}).
\cite{MuSo93} and \cite{Mue96} studied stochastic heat equations and
showed that the noises can produce blowup with positive probability.
In contrast, our blowup results depend on the
choices of initial data, it is inspired by the deterministic work of \cite{HuYi95}, \cite{LiYi14} and \cite{ln1}. A special (random) data can, with probability $1$,
lead to a blowup of the SPDEs
solutions. By increasing the amplitude of the initial data, we can get
the blowup in any short time with positive probability.
\vskip 3mm

Both probabilistic and PDE's methods play important roles in our approach.
It$\hat{o}$ formula in the proof of Lemma \ref{l:MonHxi} below is the key point for finding
the monotone stochastic process $\hat \xi(t)$, which paves the way to applying classical PDE
techniques and estimating the upper bounds of blow up time.
For the PDE's argument, we follow the approaches shown in \cite{ln1} and \cite{LiYi14}.

The organization of the paper is as follows. In section 2, we introduce some
notations and give some prerequisite lemmas. To show our approach
more transparently, we prove a blowup theorem under some additional assumption in section 3.
The 4th section removes the assumption and build the general blowup result by integral estimates.

\section{Some auxiliary lemmas and a monotone stochastic process $\hat \xi(t)$}
From now on, we assume $\mcl O=B_1(0)$, the unit open ball in $\R^n$ with zero center.
For the notatioal simplicity, write $v(t,z)=e^t u(t,z)$ for all $t>0$ and $z \in \bar B_1(0)$ and
\
{\Be   \label{e:KtDef}
K(t)=\frac{e^{-(p-1)t}}{\xi^q(t)},
\Ee}
 then
 \begin{equation} \label{e:VEqn} \left\{
\begin{array}{ll}
\p_t v=\Delta v+K(t) v^p &\mbox{ in } B_1(0),\\[2mm]
\dif {\xi}=\left(-\xi+e^{-r t} \frac{\overline{v^r}}{\xi^s}\right)\dif t+\xi \dif B_t, & \\[2mm]
\frac{\partial v}{\partial\nu}=0 & \mbox{ on }\ \{z=1\}, \\[2mm]
v(0)=u_0, \\[2mm]
  \xi(0)=\xi_0,
\end{array} \right. \end{equation}
To study the blow up of Eq. \eqref{gmshadow} , we only need to study that of Eq. \eqref{e:VEqn}. So we shall concentrate on
the blow up of $v$ and $\xi$ in the sequel.

Write
$$B^{*}_t=\sup_{0 \le s \le t} |B_s| \ \ \ \ \forall t>0,$$
it is well known (\cite[p. 96]{KaSh91}) that for any $A>0$,
\
\Bes
\PP\left(B^{*}_t \ge A\right) \le \frac{\sqrt t}{\sqrt{2 \pi}} \frac{4} A\re^{-\frac{A^2}{2t}}  \ \ \ \ \ \ \ \ \ \ \forall \ t>0.
\Ees
Hence,
\
\Bes
\PP\left(B^{*}_t<\infty\right)=1-\lim_{A \rightarrow \infty} \PP\left(B^{*}_t \ge A\right)=1 \ \ \ \ \ \ \forall \ t>0.
\Ees
For every $t>0$, denote $\mcl N_t=\{\omega: B^*_t=\infty\}$, it is clear that $\PP(\mcl N_t)=0$.
Take $t=1,2,...$, it is easy to see that $\mcl N_t \subset \mcl N_m$ for all $t \le m$. Define
$\mcl N=\lim_{m\rightarrow \infty} \mcl N_m$, we have $\PP(\mcl N)=\lim_{m \rightarrow \infty} \PP(\mcl N_m)=0$.
Hence, for all $\omega \in \Omega \setminus \mcl N$,
\
$$
B^*_t(\omega)<\infty \ \ \ \ \ \ \ \ \  \forall \ t>0.
$$
From the above observation, without loss of generality, we can assume that for all $\omega \in \Omega$,
\
\Be  \label{e:Bt*Fin}
B^{*}_t(\omega)<\infty \ \ \ \ \ \ \ \ \  \forall \ t>0.
\Ee
\vskip 3mm

%\Be \label{e:OmeN}
%\Omega_N=\left\{\omega: B^{*}_1 \le N\right\},
%\Ee
%it is clear that
%$$\lim_{N \rightarrow \infty} \PP(\Omega_N)=1.$$

For all $x \in \R^n$, denote $z=|x|$.
Consider the following isotropic function
$$\phi(z)=\begin{cases}
z^{-\alpha}, & \delta \le z \le 1, \\
\delta^{-\alpha}(1+\frac \alpha 2)-\frac \alpha 2 \delta^{-\alpha-2} z^2, & 0 \le z<\delta,
\end{cases}$$
with some $\delta \in (0,1)$ and
\Be \label{e:AlpDef}
\alpha=\frac{2}{p-1}.
\Ee
 it is easy to check that
\begin{equation} \label{e:PPhi>0}
\p^2_z \phi+\frac{n-1}{z} \p_z \phi+\alpha n \phi^p \ge 0
\end{equation}
holds for all $z \in (0,1)$.

Take
$$v_0=\gamma \phi$$
as the initial data of Eq. \eqref{e:VEqn}, where $\gamma>0$ is some (random) number. This special choice of initial data is
inspired by the deterministic work of \cite{HuYi95}, \cite{LiYi14} and \cite{ln1}.
Since the initial data is isotropic in the space,
then the solution $v(x,t)$ is
also spatially isotropic for all $t>0$. Hence, we denote the solution by $v(z,t)$ and Eq. \eqref{e:VEqn} can be
rewritten as
\
\begin{equation} \label{e:VZEqn}
\left\{
\begin{array}{ll}
\p_t v=\p^2_z v+\frac{n-1}{z}\p_z v+K(t) v^p &\mbox{ in } B_1(0),\\[2mm]
\dif {\xi}=\left(-\xi+e^{-r t}\frac{\overline{v^r}}{\xi^s}\right)\dif t+\xi \dif B_t, & \\[2mm]
\frac{\partial v}{\partial z}=0 & \mbox{ on } \ \{z=1\}, \\[2mm]
v(0)=v_0, \\[2mm]
  \xi(0)=\xi_0,
\end{array} \right. \end{equation}

\noindent By a Banach fixed point argument as in \cite{WXZZ14}, Eq. \eqref{e:VZEqn} has a unique \emph{local} solution. The next lemma is about the property of the solution.
\
\begin{lem}  \label{l:ULowBou}
Let $v$ be the solution to Eq. \eqref{e:VEqn} on $[0,T]$.
Then the following statements hold:

(i).
$v(t) \ge \gamma$ for all $0 \le t \le T$.

(ii).
$\p_z v(z,t) \le 0$ for all $0 \le t \le T$ and all $0<z<1$.

(iii). For all $\beta \in (0,1]$, we have
$z^n v^{\beta}(z,t) \le \overline{v^\beta}(t)$ for all $0 \le t \le T$ and all $0<z<1$.

(iv).
$\p_z v(\frac 12,t) \le - C_0  2^{n-1}$ for all $0 \le t \le T$, where $C_0>0$ depends on $\gamma$.
\end{lem}

\begin{proof}
The proofs of (i) and (ii) are the same as those in \cite[Lemma 2.1]{LiYi14}.
By (ii), it is easy to see
\
\Bes
\begin{split}
v^\beta(z,t) z^n&=v^\beta(z,t) \int_0^z n r^{n-1} \dif r \le \int_0^z v^\beta(r,t)n r^{n-1} \dif r \\
&  \le  \int_0^1 v^\beta(r,t)n r^{n-1} \dif r=\frac{1}{|B_1(0)|} \int_{B_1(0)} v^{\beta} (x,t) \dif x.
\end{split}
\Ees
Hence, (iii) is proved.

Now we consider $f(z,t)=z^{n-1}  \p_z v$, it is straightforward to check that
\
\Be \label{e:GMax}
\p_t f=\p_{z}^2 f-\frac{n-1}{z} \p_z f+p K(t) v^{p-1} f \ \ \ \ \ \ \ {\rm in}\  B_1(0) \times (0,T).
\Ee
It is easy to check that $f(1,t)=z^{n-1} \p_z v(z,t) |_{z=1}=0$ and that $f(z,0)<-\gamma \alpha$ for all $\frac 14<z<1$. Applying strong maximum principle to $f$, we get $f(\frac 12,t) \le -C_0$ for all $t \in (0,T)$, this immediate gives (iv).
\end{proof}

Define
$$\hat \xi(t)=e^{\frac {3t} 2-B_t} \xi(t) \ \ \ \ \ \ \ \ t>0,$$
we have the following lemma:

\begin{lem} \label{l:MonHxi}
We have
\begin{equation}
\hat \xi(t) \ge \hat \xi(s) \ \ \ \ \ \ \ \ t \ge s \ge 0.
\end{equation}
\end{lem}

\begin{proof}
By It$\hat{o}$ formula, we have
\
\begin{equation}  \label{e:HatXiRel}
\begin{split}
\dif \hat \xi(t)&=\dif \left(e^{\frac{3t}2-B_t} \xi(t) \right) \\
&=\xi(t) \dif \left(e^{\frac{3t} 2-B_t}\right)+e^{\frac{3t} 2-B_t} \dif \xi(t)+\left(\dif e^{\frac{3t} 2-B_t}\right)\bigg(\dif \xi(t)\bigg)  \\
&=\xi(t)\left[\frac32 e^{\frac {3t}2-B_t} \dif t-e^{\frac {3t}2-B_t} \dif B_t+\frac12 e^{\frac {3t}2-B_t} \dif t\right] \\
&\ \ \ \ +e^{\frac {3t}2-B_t} \left[-\xi(t) \dif t+e^{-r t}\frac{\overline{v^r}(t)}{\xi^s(t)}\dif t+\xi(t) \dif B_t \right]-e^{\frac {3t}2-B_t} \xi(t) \dif t \\
&=e^{-r t+\frac{3t}2-B_t}\frac{\overline{v^r}(t)}{\xi^s(t)}\dif t.
\end{split}
\end{equation}
Since $\overline{v^r}(t) \ge 0$ and $\xi(t) \ge 0$ for all $t \ge 0$, $\hat \xi(t)$ is an increasing function with
respect to $t$. This completes the proof.
\end{proof}

Since $\hat \xi(0)=\xi_0$, by Lemma \ref{l:MonHxi} we have $\hat \xi(t) \ge \xi_0$ for all $t \ge 0$. For any $\lambda \in (1,\infty)$, define
\
\Be \label{e:TLam}
t_{\lambda}=\inf\{t \ge 0: \hat \xi(t) \ge \lambda \xi_0\}
\Ee
with the convention $\inf \emptyset=\infty$.
($t_\lambda$ is actually a stopping time). It is easy to see that $t_\lambda=\infty$ holds as long as $\hat \xi(t)<\lambda \xi_0$ for all $t>0$.  We clearly have
\
\Be \label{e:HxiBD}
\xi_0 \le \hat \xi(t) \le \lambda \xi_0 \ \ \ \ \ \ \ \ \ \ t \in [0,t_\lambda].
\Ee
In \eqref{e:HxiBD},  we define $\hat \xi(\infty)=\lim_{t \rightarrow \infty} \hat \xi(t)$ as $t_\lambda=\infty$.
\vskip 3mm

Let $\theta: \Omega \rightarrow (0,\infty)$ be a positive random variable. From \eqref{e:Bt*Fin}, we clearly have
\Be
B^*_{\theta(\omega)}(\omega)<\infty \ \ \ \ \ \ \ \ \  \forall \ \omega \in \Omega,
\Ee
for notational simplicity, we shall suppress the variable $\omega$ and write it as $B^*_\theta$.
 Recall the definition of $K(t)$ in \eqref{e:KtDef}, we have
\
{\Be \label{e:KtBou}
(\lambda \xi_0)^{-q} \exp \left(-(p-1) \theta-q B^*_{\theta}\right)\le K(t) \le \xi_0^{-q} \exp \left({3\over 2}q\theta + q B^*_{\theta}\right),    \ t \in [0,\theta].
\Ee}
Indeed, it is easy to see that
\
{$$K(t)=\hat \xi(t)^{-q} \exp \left( -(p-1)t+\frac {3} 2qt- q B_t\right)$$}
holds. By \eqref{e:HxiBD}, we have
\
{
$$
(\lambda \xi_0)^{-q}  \exp \left( -(p-1)t+\frac {3} 2qt- q B_t\right) \le K(t) \le \xi_0^{-q}  \exp \left( -(p-1)t+\frac {3} 2qt- q B_t\right),
$$}
which immediately implies the desired \eqref{e:KtBou}. For the further usage, we denote
\
\Be  \label{e:Tb}
T_b {\rm \ \ the \ \ blowup \ \ time \ \ of \ the \ solution \ \ }v(z,t),
\Ee
\Be \label{e:KThe}
K_\theta=(\lambda \xi_0)^{-q} \exp \left(-(p-1) \theta-q B^*_{\theta}\right).
\Ee
\ \ \  \

\section{Pointwise blow up as $t_\lambda \ge \theta$} \label{s:Easy}

Let $\theta \in (0,\infty)$ be some strictly positive random variable as in the previous section. Recall the definition of $t_\lambda$ in \eqref{e:TLam} with $\lambda\in (1,\infty)$
being some fixed number, under the assumption $t_\lambda\ge \theta$, we shall prove the next two theorems,
whose proofs also partly give the main idea of our approach. The first theorem gives a upper bound of the blow up time pointwise, while
the second claims that the upper bound of the blowup time is larger than $\theta$ as $t_\lambda \ge \theta$,
which means that the blow up could happen after the time $\theta$.

Note that the quantities below such as $\tau$ and $T_b$ are random variables,
we should write them as $\tau(\omega)$ and $T_b(\omega)$ more precisely. For notational simplicity,
we shall suppress the argument $\omega$ in them if no confusions arise.
\
\begin{thm} \label{t:BUpGe1}
Let $\lambda>1$ and let $\theta \in (0,\infty)$ be some random number. If $t_\lambda \ge \theta$, choose $\gamma$ such that
$\gamma^{p-1} K_\theta>\frac{4n}{p-1}$, then
we have
\
\Be
{T_b \le  \frac{2 \delta^2}{\gamma^{p-1}K_\theta(p-1)}\left(1+{\alpha\over 2}\right)^{-p+1}<\frac{ \delta^2}{2n}\left(1+{\alpha\over 2}\right)^{-p+1}.}
\Ee
\end{thm}

\begin{proof}
By \eqref{e:KtBou}, we have
\
\Bes %\label{e:KtBou}
K(t) \ge K_\theta,  \ \ \ \ \ \ t \in [0,\theta].
\Ees
By \eqref{e:VZEqn} and the above inequality, we have
\
\begin{equation*}
\left\{
\begin{array}{ll}
\p_t v \ge \p^2_z v+\frac{n-1}{z}\p_z v+K_\theta v^p &\mbox{ in } B_1(0) \times (0,\theta),\\[3mm]
\partial_z v=0 & \mbox{ on }\{z=1\} \times (0,\theta), \\[3mm]
v(0)=v_0  & \mbox{ in } B_1(0).
\end{array} \right. \end{equation*}
Now consider another equation
\
\begin{equation} \label{e:WEqnS2}
\left\{
\begin{array}{ll}
\p_t w = \p^2_z w+\frac{n-1}{z}\p_z w+K_\theta w^p   &\mbox{ in } B_1(0) \times (0,\theta),\\[3mm]
{\partial_z w}=0 & \mbox{ on } \{z=1\} \times (0,\theta), \\[3mm]
w(0)=v_0  & \mbox{ in } B_1(0).
\end{array} \right. \end{equation}
By comparison principle, we have
$$v(z,t) \ge w(z,t), \ \ \ \ \ \ (z,t) \in B_1(0) \times [0,\theta].$$
Write $\rho=\p_t w-\frac{K_\theta}2 w^p$, a straightforward calculation gives
\
\Bes %\label{e:KtBou}
\begin{split}
\p_t \rho&=\Delta \rho+\frac{K_\theta} 2 p(p-1) w^{p-1} |\nabla w|^2+\frac{K_\theta} 2 p w^{p-1} \Delta w+\frac{K_\theta}2 p w^{p-1} \p_t w \\
& \ge \Delta \rho+\frac{K_\theta} 2 p w^{p-1} \Delta w+\frac{K_\theta}2 p w^{p-1} \p_t w \\
& =\Delta \rho+\frac{K_\theta} 2 p w^{p-1} \left(\p_t w-K_\theta w^p\right)+\frac{K_\theta}2 p w^{p-1} \p_t w \\
&=\Delta \rho+K_\theta p w^{p-1} \rho,
\end{split}
\Ees
where the second '$=$' above is by \eqref{e:WEqnS2}.
It is straightforward to check that for all $z \in B_1(0)$,
\
\Bes
\begin{split}
\rho(z,0)&=\p^2_z u_0+\frac{n-1}{z}\p_z u_0+\frac{K_\theta}2 u^p_0 \\
&=\gamma\left[\p^2_z \phi(z)+\frac{n-1}{z}\p_z \phi(z)+\frac{K_\theta}2 \gamma^{p-1} \phi^p(z)\right].
\end{split}
\Ees
Under the condition in the theorem, \eqref{e:PPhi>0} holds and thus the term
in the square bracket is positive. Therefore,
\
$$\rho(z,0) \ge 0, \ \ \  z \in B_1(0).$$
It is easy to check
\
$${\p_z \rho}=0, \ \ \ \ (z,t) \in \{z=1\} \times [0,\theta].$$
Hence, the maximum principle gives
\
$$\rho(z,t) \ge 0, \ \ \ \ \ (z,t) \in B_1(0) \times [0,\theta].$$
That is
\
\Bes %\label{e:KtBou}
\begin{split}
\p_t w-\frac{K_\theta}2 w^p \ge 0,  \ \ \ \ \ \ \ (z,t) \in B_1(0) \times [0,\theta].
\end{split}
\Ees
which implies
\
\Be \label{e:WztBou}
\begin{split}
w(z,t) & \ge \left[\frac{1}{v_0^{-p+1}(z)-\frac{K_\theta (p-1) t}2}\right]^{\frac{1}{p-1}}.
%\\
%& \ge \left[\frac{1}{1-\frac{K_ (p-1) t}2}\right]^{\frac{1}{p-1}}
\end{split}
\Ee
By the form of $v_0(z)=\gamma \phi(z)$, for every $z \in (0,1)$ the term on the right hand side \eqref{e:WztBou}
blows up at $t=\tau(z)$ with
$$
\tau(z):=\left\{
\begin{array}{ll}
 \frac{2}{K_\theta(p-1)}\gamma^{-p+1} \left[1+\frac{1-\left(\frac z \delta\right)^2}{2} \alpha\right]^{-p+1} \delta^{2} & z \in [0, \delta],\\[3mm]
\frac{2}{K_\theta(p-1)} \gamma^{-p+1} z^2 & z \in (\delta,1),
\end{array} \right.
$$
where we have used the relation $\alpha(p-1)=2$ (see \eqref{e:AlpDef}).
It is easy to see that $\tau(z)$ is an increasing function and {$\tau(0)= \frac{2 \delta^2}{\gamma^{p-1}K_\theta(p-1)}\left(1+{\alpha\over 2}\right)^{-p+1}$}, thus we get
the desired bound for $T_b$.
\end{proof}

\begin{corollary}
Assume that $\theta \le \theta_0$ a.s. with $\theta_0>0$ being some constant and that $\gamma>0$ is some
(sufficiently large) deterministic number, then we have
\
\
\Be
{\PP \left(T_b \le \frac{ \delta^2}{2n}\left(1+{\alpha\over 2}\right)^{-p+1}\right) \ge 1- \frac{\sqrt{\theta_0}}{\sqrt{2 \pi}} \frac{4}{A_0} \re^{-\frac{A_0^2}{2\theta_0}}}
\Ee
with {$A_0=\frac 1 q \ln \frac{(p-1) \gamma^{p-1}}{4n(\lambda \xi_0)^q}-\frac{p-1}{q}\theta_0$.}
\end{corollary}

\begin{proof}
By  Theorem \ref{t:BUpGe1}, it suffices to prove that
\
\Be
\PP \left(\gamma^{p-1} K_\theta>\frac{4n}{p-1}\right) \ge 1- \frac{\sqrt{\theta_0}}{\sqrt{2 \pi}} \frac{4}{A_0} \re^{-\frac{A_0^2}{2\theta_0}}.
\Ee
Since $K_\theta$ is an decreasing function of $\theta$ and $\theta \le \theta_0$ a.s., we have
\
\Be
\begin{split}
\PP \left(\gamma^{p-1} K_{\theta}>\frac{4n}{p-1}\right)& \ge \PP \left(\gamma^{p-1} K_{\theta_0}>\frac{4n}{p-1}\right) \\
& ={\PP \left(B^{*}_{\theta_0}<\frac 1 q \ln \frac{(p-1) \gamma^{p-1}}{4n(\lambda \xi_0)^q}-\frac{p-1}{q}\theta_0\right)} \\
& ={1-\PP \left(B^{*}_{\theta_0} \ge \frac 1 q \ln \frac{(p-1) \gamma^{p-1}}{4n(\lambda \xi_0)^q}-\frac{p-1}{q}\theta_0\right)} \\
& \ge 1- \frac{\sqrt{\theta_0}}{\sqrt{2 \pi}} \frac{4}{A_0} \re^{-\frac{A_0^2}{2\theta_0}}
\end{split}
\Ee
with {$A_0=\frac 1 q \ln \frac{(p-1) \gamma^{p-1}}{4n(\lambda \xi_0)^q}-\frac{p-1}{q}\theta_0$.}
\end{proof}

\begin{corollary} \label{c:BlUp}
Assume that the conditions in Theorem \ref{t:BUpGe1} hold.
Let $\gamma \rightarrow \infty$ a.s., then we have
$$T_b \rightarrow 0, \ \ \ \ a.s..$$
\end{corollary}

\begin{proof}
By Theorem \ref{t:BUpGe1}, we have
\
\
\Bes
{T_b \le  \frac{2 \delta^2}{\gamma^{p-1}K_\theta(p-1)}\left(1+{\alpha\over 2}\right)^{-p+1}.}
\Ees
As $\gamma \rightarrow \infty$ a.s.,
we get {$\frac{2 \delta^2}{\gamma^{p-1}K_\theta(p-1)}\left(1+{\alpha\over 2}\right)^{-p+1}$} a.s. and thus
$T_b \rightarrow 0 \ a.s..$
\end{proof}

\section{General pointwise blow up result}
{Recall that $T_b$ is the blowup time of $v(z,t)$  and the $K_\theta$ is defined in \eqref{e:KThe},
in this section, we shall prove the following blow up theorem:}
\
\begin{thm} \label{t:GBlUp}
Let $\lambda>1$ and let $p \ge r$ and $\frac{p-1}{r}>\frac{2}{n+2}$. We have the following two statements:

(i) In the case $t_\lambda \ge 1$, choose $\gamma>0$ such that $\gamma^{p-1} K_1>\frac{4n}{p-1}$, we have
\
\Be
{T_b \le \frac{2 \delta^2}{\gamma^{p-1}K_1(p-1)}\left(1+{\alpha\over 2}\right)^{-p+1}}
\Ee

(ii) In the case $t_\lambda \le 1$, there exists some $\hat \theta \in (0,1]$ such that as long as $\gamma^{p-1} K_{\hat \theta}>\frac{4n}{p-1}$, we have
\
\Be
{T_b \le \frac{2 \delta^2}{\gamma^{p-1}K_{\hat{\theta}}(p-1)}\left(1+{\alpha\over 2}\right)^{-p+1}.}
\Ee
\end{thm}
\vskip 3mm

By the same argument as showing Corollary \ref{c:BlUp}, we immediately get the following
corollary.
\begin{corollary}
Assume that the conditions in Theorem \ref{t:GBlUp} hold.
Let $\gamma \rightarrow \infty$ a.s., then we have
$$T_b \rightarrow 0, \ \ \ \ a.s..$$
\end{corollary}
\vskip 3mm

Let $\beta \in (0,1]$ be some number to be determined later. Denote
$$h(t)=\overline{v^\beta}(t)\ \ \ \ \ \ t>0.$$
For the further usage, we define
\Be
h_1(t)=\frac{1}{|B_1(0)|} \int_{B_R(0)} v^{\beta}(z,t) \dif z,
\Ee
\Be
h_2(t)=\frac{1}{|B_1(0)|} \int_{B_1(0) \setminus B_R(0)} v^{\beta}(z,t) \dif z.
\Ee
where $R \in (0,1)$ is some number to be determined later.
We also define the following stochastic quantity:
\
\Be  \label{e:HLam}
h^*_\lambda=h(0)+\beta (\lambda-1) \lambda^{-q} \gamma^{\beta+p-1-r}\exp\left(-p t_\lambda-(s+q+1) (\frac {3t_\lambda}2+B^*_{t_\lambda})\right) \xi_0^{s-q+1},
\Ee
it will frequently appear in the arguments below. It is easy to see
\
\Be  \label{e:HLamUp}
h^*_\lambda \le h(0)+\beta (\lambda-1) \lambda^{-q} \gamma^{\beta+p-1-r} \xi_0^{s-q+1}.
\Ee

\begin{lem}
Let $\lambda>1$ and $p \ge r$. Assume $t_\lambda<\infty$. Choose $\beta \in (0,1]$ such that $p+\beta-1 \ge r$, then we have
\Bes
\begin{split}
h(t_\lambda)  \ge h^*_\lambda.
\end{split}
\Ees
\end{lem}

\begin{proof}
We have
\
\Be
\begin{split}
\frac{\dif h(t)}{\dif t}&=\frac{\beta}{|B_1(0)|} \int_{B_1(0)}  v^{\beta-1}(x,t) \p_t v(x,t) \dif x  \\
&=\frac{\beta}{|B_1(0)|} \int_{B_1(0)}  v^{\beta-1}(x,t) \left[\Delta v(x,t)+ K(t) v^p(x,t)\right] \dif x  \\
&=\frac{\beta(1-\beta)}{|B_1(0)|} \int_{B_1(0)}  v^{\beta-2}(x,t) |\nabla v(x,t)|^2 \dif x
+\beta K(t) \overline{v^{\beta+p-1}}(t)  \\
&\ge \beta K(t) \overline{v^{\beta+p-1}}(t) \\
\end{split}
\Ee
where the last inequality is by $\beta \in (0,1]$. Since $p\ge r$ and $\beta \in (0,1]$ are such that
$p+\beta-1 \ge r$, by Lemma \ref{l:ULowBou} (i), we have $v(z,t) \ge \gamma$ for all $t>0$ and $0<z<1$ and thus
\
$$
\overline{v^{\beta+p-1}}(t) \ge \gamma^{\beta+p-1-r}\overline{v^{r}}(t).
$$
Hence,
\
\Bes
\frac{\dif h(t)}{\dif t} \ge \beta \gamma^{\beta+p-1-r} K(t) \overline{v^{r}}(t)
\Ees
\noindent
On the other hand, by \eqref{e:HatXiRel}, we have
\
\Bes
\overline{v^r}(t)=e^{r t-\frac{3t}2+B_t} \xi^s(t) \frac{\dif \hat \xi(t)}{\dif t}.
\Ees
Hence, by \eqref{e:KtDef}, \eqref{e:HxiBD} and the above relations, we have
\
\Bes
\begin{split}
\frac{\dif h(t)}{\dif t} & \ge  \beta \gamma^{\beta+p-1-r}K(t)  e^{r t-\frac {3t}2+B_t} \xi^s(t) \frac{\dif \hat \xi(t)}{\dif t} \\
&= \beta \gamma^{\beta+p-1-r} \frac{{\hat \xi}^{s}(t)}{{\hat \xi}^{q}(t)}  e^{(r-p+1) t} e^{(s-q+1)(-\frac{3t}2+B_t)}  \frac{\dif \hat \xi(t)}{\dif t} \\
& \ge \beta \gamma^{\beta+p-1-r} \lambda^{-q} \xi_0^{s-q}  e^{-pt} e^{-(s+q+1) (\frac {3t}2+B^*_t)}  \frac{\dif \hat \xi(t)}{\dif t}\\
& \ge \beta \gamma^{\beta+p-1-r} \lambda^{-q} \xi_0^{s-q}  \exp\left(-p t_\lambda-(s+q+1) (\frac {3t_\lambda}2+B^*_{t_\lambda})\right) \frac{\dif \hat \xi(t)}{\dif t}
\end{split}
\Ees
for all $t \in [0,t_\lambda]$.
By the definition of $t_\lambda$ and Lemma \ref{l:MonHxi}, we immediately get the desired inequality.
\end{proof}

Stimulated from the previous lemma, we define
\Be  \label{e:T2Def}
\begin{split}
\hat t_\lambda=\inf\{t \ge 0: h(t)  \ge h^*_\lambda\},
\end{split}
\Ee
it is clear that
$$\hat t_\lambda \le t_\lambda$$
and
\
\Be   \label{e:t2UpBou}
\begin{split}
h(t)  \le h^*_\lambda,  \ \ \  \ \ \ \ \ \ \ \  t \in [0,\hat t_\lambda].
\end{split}
\Ee

\noindent Denote $f(z,t)=z^{n-1} \p_z v(z,t)$, it is easy to check
\
$$
\mcl L f=\p_t f-\p^2_z f+\frac{n-1}{z} \p_z f-p K(t)  v^{p-1} f=0.
$$
The proof of the next lemma has some similarity to that of \cite[Lemma 2.2]{FrMc85}.
\
\begin{lem}  \label{l:VztIn}
Let $\lambda>1$. Let $k \in (1,p)$, $\beta \in (0,1]$ and $\ell \ge \frac k{\beta}$. Assume $t_\lambda<\infty$. As $\e \le \e^*$ with
\Be   \label{e:EpsSel}
{\e^*=\min\left\{\alpha (1+\frac \alpha 2)^{-k} h^\ell(0), \ \ 2^{-\ell n+n} C_0 \gamma^{\beta \ell-k}, \ \frac{(p-k)\gamma^{p+\ell-k}}{2k (\lambda \xi_0)^q} \exp\left(-(p-1)t_{\lambda}-q B^{*}_{t_\lambda}\right)\right\},}
\Ee
we have
\Be
v(z,t) \le \left(\frac{2 h^{\ell}(t)}{\e(k-1)}\right)^{\frac{1}{k-1}} z^{-\frac{2}{k-1}}, \ \
\ \ \ \ \ \ \forall \ t \in [0,\hat t_\lambda] \ \ \forall \ z \in (0,\frac 12].
\Ee
\end{lem}

\begin{proof}
Denote $\eta(z,t)=f(z,t)+\e z^n \frac{v^k(z,t)}{h^{\ell}(t)}$ with $f(z,t)=z^{n-1} \p_z v(z,t)$ and $\e>0$ some number to be
determined later and $\ell \ge \frac{k}{\beta}$,
we prove the lemma in the following three steps.

\emph{Step 1: Property of $\eta(z,t)$}
 By (i), (iii) and (iv) of Lemma \ref{l:ULowBou} and the relation $\ell  \ge \frac k\beta$,
we further have
\
\Be \label{e:1/2Bou}
\begin{split}
\eta(\frac 12,t)&=(\frac 12)^{n-1} \p_z v(\frac 12,t)+\e (\frac 12)^n \frac{v^k(\frac 12,t)}{h^\ell(t)}   \\
& \le -C_0+\e (\frac 12)^n \left(\frac{v^\beta(\frac 12,t)}{h(t)}\right)^{\ell}  v^{k-\beta \ell} (\frac 12,t)  \\
& \le -C_0+\e (\frac 12)^n \left(\frac{v^\beta(\frac 12,t)}{h(t)}\right)^{\ell}  \gamma^{k-\beta \ell}   \\
& \le -C_0+\e 2^{n\ell-n} \gamma^{k-\beta \ell} \\
\end{split}
\Ee

\noindent As $t=0$, for all $z \in (0,\delta)$, by the relation $\alpha+2=p \alpha>k \alpha$, we have
\
\Be   \label{e:IniCon1}
\begin{split}
\eta(z,0) & \le \left[-\alpha \delta^{-\alpha-2}+\e (1+\frac \alpha 2)^k \frac{1}{h^\ell(0)} \delta^{-\alpha k}\right] z^n \\
& = \left[-\alpha+\e (1+\frac \alpha 2)^k \frac{1}{h^\ell(0)} \delta^{\alpha+2-\alpha k}\right] \delta^{-\alpha-2} z^n \\
& \le \left[-\alpha+\e (1+\frac \alpha 2)^k \frac{1}{h^\ell(0)} \right] \delta^{-\alpha-2} z^n.
\end{split}
\Ee

\noindent
For all $z \in (\delta,1)$, by the relation $\alpha+2=p \alpha>k \alpha$ again, we have
\
\Be  \label{e:IniCon2}
\begin{split}
\eta(z,0) & \le  \left[-\alpha+\frac{\e}{h^\ell(0)} z^{\alpha+2-\alpha k}\right] z^{n-\alpha-2} \\
& \le  \left[-\alpha+\frac{\e}{h^\ell(0)}\right] z^{n-\alpha-2}.
\end{split}
\Ee

\noindent Hence, collecting \eqref{e:1/2Bou}-\eqref{e:IniCon2}, as long as
\Be
\e \le \min\left\{\alpha (1+\frac \alpha 2)^{-k} h^\ell(0), \ \ 2^{-\ell n+n} C_0 \gamma^{\beta \ell-k}\right\},
\Ee
we have
\
\Be \label{e:Ini<0}
\eta(z,0) \le 0, \ \ \ \  \ \ \ \ \ z \in (0,1),
\Ee
\Be   \label{e:1/2<0}
\eta(\frac 12,t) \le 0, \ \ \ \  \ \ \ \ \ t \in (0,\hat t_\lambda).
\Ee

\emph{Step 2:}  Observe
\
\Bes
\begin{split}
\mcl L \eta&=\mcl L \left(\e z^n \frac{v^k}{h^{\ell}}\right) \\
&=-2 \e k z^{n-1} \frac{v^{k-1}}{h^\ell} \p_z v-\e(p-k) {e^{-(p-1)t}} \frac{z^n}{\xi^q} \frac{v^{p-1+k}}{h^\ell}
-\e k(k-1)z^n \frac{v^{k-2}}{h^{\ell}} (\p_z v)^2 \\
&\ \ \ -\e \beta(1-\beta) \ell z^n \frac{v^k}{h^{\ell+1}} \overline{v^{\beta-2} |\nabla v|^2}
-\e \ell \beta {e^{-(p-1)t}}\frac{z^n}{\xi^q}  \frac{v^k}{h^{\ell+1}} \overline{v^{\beta+p-1}} \\
& \le -2 \e k z^{n-1} \frac{v^{k-1}}{h^\ell} \p_z v-\e(p-k){e^{-(p-1)t}} \frac{z^n}{\xi^q} \frac{v^{p-1+k}}{h^\ell} \\
&=-2 \e k  \frac{v^{k-1}}{h^\ell} \eta+\frac{\e z^n v^k}{h^{2\ell}} \left[2 \e k v^{k-1}-(p-k){e^{-(p-1)t}}
\frac{h^{\ell} v^{p-1}}{\xi^q}\right].
\end{split}
\Ees
Recall that $v(t) \ge \gamma$ for all $t \ge 0$ from Lemma \ref{l:ULowBou} and that
$\xi_0 \le \hat \xi(t) \le \lambda \xi_0$
for all $t \in [0,t_\lambda]$ where $\hat \xi(t)=e^{\frac {3t}2-B_t} \xi(t)$, we have

\Bes
\begin{split}
{e^{-(p-1)t}}\frac{h^{\ell}(t) v^{p-k}(t)}{\xi^q(t)}&={e^{-(p-1)t}}\frac{h^{\ell}(t) v^{p-k}(t)}{{\hat \xi}^q(t) e^{-\frac q2 t+q B_t}}
\ge {e^{-(p-1)t}}\frac{\gamma^{p+\ell-k}}{{\hat \xi}^q(t) e^{-\frac q2 t+q B_t}} \\
& \ge {e^{-(p-1)t}}\frac{\gamma^{p+\ell-k}}{(\lambda \xi_0)^q e^{-\frac q2 t+q B_t}} \ge  \frac{e^{{-(p-1)t_{\lambda}}-q B^*_{t_\lambda}}\gamma^{p+\ell-k}}{(\lambda \xi_0)^q}, \ \ \ \  \ \ \ \ \ \ \ \  \ \ \ \ t \in [0,t_\lambda].
\end{split}
\Ees
Hence, as long as
\Be \label{e:ELes}
\e \le \frac{(p-k)e^{{-(p-1)t_{\lambda}}-q B^*_{t_\lambda}}\gamma^{p+\ell-k}}{2k (\lambda \xi_0)^q},
\Ee
we have
\Be   \label{e:LLes}
\mcl L \eta \le -2 \e k  \frac{v^{k-1}}{h^\ell} \eta.
\Ee

\emph{Step 3:}  Choose
$$\e \le \min\left\{\alpha (1+\frac \alpha 2)^{-k} h^\ell(0), \ \ 2^{-\ell n+n} C_0 \gamma^{\beta \ell-k}, \ \frac{(p-k)\gamma^{p+\ell-k}}{2k (\lambda \xi_0)^q} \exp\left({-(p-1)t_{\lambda}}-q B^{*}_{t_\lambda}\right)\right\} ,$$
then \eqref{e:LLes}, \eqref{e:Ini<0} and \eqref{e:1/2<0} all hold. By comparison principle, we get
$$\eta(z,t) \le 0,  \ \ \ \ \ \ \ \ \forall \ 0<z<1 \ \ \ \forall \ 0 \le t \le \hat t_\lambda.$$
This implies
\ \ \ \ \ \ \ \ \ \
\Bes
\p_z v(z,t) \le -\e  z \frac{v^{k} (z,t)}{h^\ell(t)},
\Ees
thus
\ \ \ \ \ \
\Be
v(z,t) \le \left(\frac{2 h^{\ell}(t)}{\e(k-1)}\right)^{\frac{1}{k-1}} z^{-\frac{2}{k-1}}, \ \
\ \ \ \ \ \ \forall \ t \in [0,\hat t_\lambda] \ \ \forall \ z \in (0,\frac 12].
\Ee
\end{proof}

\begin{lem} \label{l:GraV}
Assume $t_\lambda \le 1$. For $R \in (0,1)$, we have
\Be
 |\p_z v(z,t)| \le C_1, \ \ \ \ \ \ \ \forall \ z \in [R,1] \ \ \ \forall \ t \in [0,\hat t_\lambda],
\Ee
where $C_1$ is some number depending on $R$ and $\gamma$.
\end{lem}
\begin{rem}
In the lemma, we assume $t_\lambda \le 1$, the $1$ here can be replaced by any other positive number.
It seems that the assumption $t_\lambda \le 1$ is necessary for getting the bound $C_1$ which only depends on $R$.
\end{rem}

\begin{proof}
Writing $w(z,t)=\p_z v(z,t)$, by Eq. \eqref{e:VZEqn} we have
\
\Be \label{e:WEqn}
\p_t w=\p^2_z w+\frac{n-1}{z} \p_z w+\left(p K(t) v^{p-1}-\frac{n-1}{z^2}\right) w.
\Ee
By (iii) of Lemma \ref{l:ULowBou} and \eqref{e:t2UpBou}, we have
\
\Bes
v^{\beta}(z,t) \le z^{-n} h(t) \le z^{-n} h^*_\lambda, \ \ \ \ \ \forall \ t \in [0,\hat t_\lambda].
\Ees
This and \eqref{e:HLamUp} further implies
\Be  \label{e:VztUpBou}
\begin{split}
v^{\beta}(z,t) \le R^{-n} \left[h(0)+\beta (\lambda-1) \gamma^{\beta+p-1-r} \lambda^{-q}  \xi_0^{s-q+1}\right], \ \ \ \ \ \forall \ t \in [0,\hat t_\lambda], \ z \in [R,1].
\end{split}
\Ee
Since $\hat t_\lambda \le t_\lambda \le 1$, by \eqref{e:KtBou}, we have
\Be
{K(t) \le \xi_0^{-q} \exp \left({3\over 2}q+ q B^*_{1}\right),  \ \ \ \ \ \ t \in [0,\hat t_\lambda].}
\Ee
We can extend Eq. \eqref{e:WEqn} from the time interval $[0,\hat t_\lambda]$ to $[0,1]$ by
\Bes
K(t) v^{p-1}(z,t)=K(\hat t_\lambda) v^{p-1}(z, \hat t_\lambda), \ \ \ \ \ \ \ \forall \ z \in [R,1) \ \ \forall \ t \in [\hat t_\lambda, 1].
\Ees
Now Eq. \eqref{e:WEqn} with $(z,t) \in [R,1) \times [0,1]$ have uniformly bounded coefficients.

On the other hand, as $t=0$, it is easy to check
$$|\p_z \phi(z)|=\alpha z^{-\alpha-1},  \ \ \ \ \ \ z \in [\delta,1],$$
$$|\p_z \phi(z)|=\alpha \delta^{-\alpha-2}z,  \ \ \ \ \ \ z \in [0,\delta].$$
Indeed, if $R \ge \delta$, then the first identity above implies
\
\Be  \label{e:DPhiBou}
{|\p_z v_0(z)| \le \alpha\gamma  R^{-\alpha-1}},  \ \ \ \ \ \ z \in [R,1];
\Ee
if $R<\delta$, then the second identity above implies \eqref{e:DPhiBou} as well. Hence,
$$|\p_z v_0(z)| \le \alpha \gamma R^{-\alpha-1}, \ \ \ \ \ \ \ \ z \in [0,1].$$

\noindent So, by parabolic regularity (\cite{Lie96}), we immediately get the desired inequality.
\end{proof}

\begin{lem}   \label{l:HatT}
Assume $t_\lambda \le 1$. Let $p \ge r$ and $\frac{p-1}{r}>\frac{2}{n+2}$.
Let $\beta \in (0,1]$ be such that $p+\beta-1 \ge r$ holds.
For any $R \in (0,1)$,
we have
\  \
\Be  \label{e:BloTimEst}
\begin{split}
\hat t_\lambda \ge \frac{h^*_\lambda-h(0)-n \left(\frac{2 (h^*_\lambda)^{\ell}}{\e^* (k-1)}\right)^{\frac{1}{k-1}}
\frac{R^{n-\frac{2\beta}{k-1}}}{n-\frac{2 \beta}{k-1}}}{L(\beta,C_1,\lambda,\gamma, p,q,R)},
\end{split}
\Ee
where $k \in (1,p)$, $\ell \ge \frac k{\beta}$, $\e^*$ is defined by \eqref{e:EpsSel}, and
{\begin{eqnarray*}
&& L(\beta,C_1,\lambda, \gamma,p,q,R)\\
&:=&C_1 n\beta R^{n-1}\gamma^{\beta-1}
+ C_1^2\beta (1-\beta )  \gamma^{\beta-2}  +  \beta \xi_0^{-q} \exp\left({3\over 2 }qt_{\lambda} + qB^*_{t_{\lambda}}\right)  \left(\frac{h^*_{\lambda}}{R^{n}}\right)^{\beta+p-1\over\beta}
\end{eqnarray*}}
with $C_1$ being the number in Lemma \ref{l:GraV} (which depends on $R$).
\end{lem}

\begin{rem}
We can tune the number $R$ such that the right hand of \eqref{e:BloTimEst} is strictly large than $0$ and
make the claim $\hat t_\lambda>\hat \theta>0$ be true.
\end{rem}

\begin{proof}
Recall that
$$h(t)=h_1(t)+h_2(t)$$
where
\Bes
h_1(t)=\frac{1}{|B_1(0)|} \int_{B_R(0)}  v^{\beta}(x,t) \dif x, \ \ \ \ \ \ \ \ h_{2}(t)=\frac{1}{|B_1(0)|} \int_{B_1(0) \setminus B_R(0)}  v^{\beta}(x,t) \dif x.
\Ees
with $R$ being some number to be chosen. By Lemma \ref{l:VztIn}, we have
\
\Be  \label{e:VUpBou4.6}
v(z,t) \le \left(\frac{2 h^{\ell}(t)}{\e^*(k-1)}\right)^{\frac{1}{k-1}} z^{-\frac{2}{k-1}}, \ \
\ \ \ \ \ \ \forall \ t \in [0,\hat t_\lambda] \ \ \forall \ z \in (0,\frac 12].
\Ee
Since $\frac{p-1}{r}>\frac{2}{n+2}$,
we have
\
$$n(p-1)>2(r+1-p).$$
Thanks to the condition $p+\beta-1 \ge r$ with $\beta \in (0,1]$, we can choose some $\beta \in (0,1]$ so that
\
$$n(p-1)>2 \beta \ge 2(r+1-p).$$

\noindent Therefore, we can choose some $k \in (1,p)$ so that
$$n(k-1)>2 \beta.$$
Hence, for any $t \in [0,\hat t_\lambda]$, by \eqref{e:VUpBou4.6} and \eqref{e:t2UpBou}, we have
\
\Be  \label{e:H1Est}
\begin{split}
h_1(t)&=n \int_0^R  v^{\beta}(z,t) z^{n-1} \dif z \\
& \le n \left(\frac{2 h^{\ell}(t)}{\e^*(k-1)}\right)^{\frac{\beta}{k-1}} \int_0^R z^{n-1-\frac{2\beta}{k-1}} \dif z \\
&=n \left(\frac{2 h^{\ell}(t)}{\e^*(k-1)}\right)^{\frac{\beta}{k-1}} \frac{R^{n-\frac{2\beta}{k-1}}}{n-\frac{2 \beta}{k-1}} \\
& \le n \left(\frac{2 (h^*_\lambda)^{\ell}}{\e^*(k-1)}\right)^{\frac{\beta}{k-1}} \frac{R^{n-\frac{2\beta}{k-1}}}{n-\frac{2 \beta}{k-1}}.
\end{split}
\Ee

\noindent Now we consider $h_2(t)$, by \eqref{e:VEqn}, it is easy to see
\
{\begin{eqnarray*}
{d\over dt}h_2(t)
&=&{d\over dt}{1\over |B_1(0) |}\int_{B_1(0)\setminus B_R(0)}v^{\beta}(x,t)dx\\
&=&{\beta\over |B_1(0) |}\int_{B_1(0)\setminus B_R(0)} v^{\beta-1} \left(\Delta v + K(t)v^p\right)dx\\
&=& -n\beta R^{n-1}v^{\beta-1}(R,t)v_z(R,t)
-{\beta (\beta-1)\over |B_1(0) |}\int_{B_1(0)\setminus B_{R}(0)} v^{\beta-2} |\nabla v |^2 dx\\
& & +  \beta K(t) {1\over |B_1(0) |}\int_{B_1(0)\setminus B_{R}(0)} v^{\beta+p-1}dx.
\end{eqnarray*}}
By (i) of Lemma \ref{l:ULowBou} and Lemma \ref{l:GraV}, we have $v^{\beta-1} \le \gamma^{\beta-1}$, $v^{\beta-2} \le \gamma^{\beta-2}$ and
\
{\begin{eqnarray*}
\left|{d\over dt}h_2(t)\right|
&\leq & n\beta R^{n-1}\gamma^{\beta-1}(R,t) |v_z(R,t)|
+ {\beta (1-\beta )\over |B_1(0) |} \gamma^{\beta-2}\int_{B_1(0)\setminus B_{R}(0)}  |\nabla v |^2 dx\\
& & +  \beta \frac{e^{-(p-1)t}}{\xi^q(t)} {1\over |B_1(0) |}\int_{B_1(0)\setminus B_{R}(0)} v^{\beta+p-1}dx.
\end{eqnarray*}}
By (iii) of Lemma \ref{l:ULowBou}, we have $v^\beta(z,t) \le \frac{h(t)}{z^{n}}$. This and \eqref{e:t2UpBou} further give
\
{\begin{eqnarray}\label{e:DH2Est}
\left|{d\over dt}h_2(t)\right|
&\leq & C_1 n\beta R^{n-1}\gamma^{\beta-1}
+ C_1^2\beta (1-\beta )  \gamma^{\beta-2} \cr
& & +  \beta  e^{-(p-1)t}  \left(\hat{\xi}(t)e^{-{3\over 2}t+B_t}\right)^{-q} {1\over |B_1(0) |}\int_{B_1(0)\setminus B_{R}(0)} \left(\frac{h(t)}{z^{n}}\right)^{\beta+p-1\over\beta}dx\cr
&\leq & C_1 n\beta R^{n-1}\gamma^{\beta-1}
+ C_1^2\beta (1-\beta )  \gamma^{\beta-2} \cr
& & +  \beta \xi_0^{-q} \exp\left({3\over 2 }qt_{\lambda} + qB^*_{t_{\lambda}}\right)  \left(\frac{h^*_{\lambda}}{R^{n}}\right)^{\beta+p-1\over\beta} \\
& :=& L(\beta, C_1,\lambda,\gamma,p,q, R )\nonumber
\end{eqnarray}}
for all $t \in [0,\hat t_\lambda]$.
\

 By the definition of $\hat t_\lambda$, \eqref{e:H1Est} and \eqref{e:DH2Est}, we have
\
\Bes
\begin{split}
h^*_\lambda-h(0) & \le h(\hat t_\lambda)-h(0) \\
& \le h_1(\hat t_\lambda)+h_2(\hat t_\lambda)-h_2(0) \\
& \le n \left(\frac{2 (h^*_\lambda)^{\ell}}{\e^*(k-1)}\right)^{\frac{1}{k-1}} \frac{R^{n-\frac{2\beta}{k-1}}}{n-\frac{2 \beta}{k-1}}
+\int_0^{\hat t_\lambda} \left|\frac{\dif}{\dif s} h_2(s)\right| \dif s \\
& \le n \left(\frac{2 (h^*_\lambda)^{\ell}}{\e^*(k-1)}\right)^{\frac{1}{k-1}} \frac{R^{n-\frac{2\beta}{k-1}}}{n-\frac{2 \beta}{k-1}}
+\hat t_\lambda L(\beta,C_1,\lambda,\gamma, p,q,R).
\end{split}
\Ees
This immediately implies the desired inequality.
\end{proof}

\begin{proof} [Proof of Theorem \ref{t:GBlUp}]
To prove the theorem, we shall consider the two cases: (i) the case $t_{\lambda} \ge 1$ and (ii) the case $t_\lambda < 1$.

(i) \emph{$t_{\lambda} \ge 1$}. Take $\theta=1$ in Section \ref{s:Easy}, we immediately get the desired estimate by Theorem \ref{t:BUpGe1}.

(ii) \emph{$t_{\lambda}<1$}. By \eqref{e:HLam}, it is easy to see that if $t_\lambda < 1$ we have
\
\Be \label{e:h-h0}
h^*_\lambda \ge h(0)+\beta (\lambda-1) \lambda^{-q} \gamma^{p+\beta-1-r} \exp\left(-p-(s+q+1) (\frac {3}2+B^*_1)\right) \xi_0^{s-q+1}.
\Ee
Recalling \eqref{e:HLamUp} as below:
\
\Be  \label{e:hup}
h^*_\lambda \le h(0)+\beta (\lambda-1) \gamma^{p+\beta-1-r} \lambda^{-q}  \xi_0^{s-q+1}.
\Ee
The estimate \eqref{e:BloTimEst}, together with \eqref{e:h-h0} and \eqref{e:hup},
implies that there exists some $R \in (0,1)$ (which can be tuned according to $p,q, \lambda,B^*_1, s,\lambda, \gamma, \beta, \xi_0$) and
some some $\hat \theta$ (depending on $\beta, p,q, \lambda, \gamma, B^*_1, s, R$)
such that
\  \
\Bes
\begin{split}
\hat t_\lambda \ge \hat \theta>0.
\end{split}
\Ees
$\hat \theta \in (0,1)$ is obvious. Since $t_\lambda \ge \hat t_\lambda$, we have $t_\lambda \ge \hat \theta$. Now we can use Theorem \ref{t:BUpGe1} to get the desired result.
\end{proof}

\end{document}